\documentclass[a4papersize]{article}

\usepackage{amsmath,amsthm, amsxtra,amssymb,latexsym, amscd, enumerate}
\usepackage[mathscr]{eucal}
\theoremstyle{plain}
\newtheorem{theorem}{Theorem}[section]
\newtheorem{corollary}[theorem]{Corollary}
\newtheorem{lemma}[theorem]{Lemma}
\newtheorem{proposition}[theorem]{Proposition}

\theoremstyle{definition}
\newtheorem{example}[theorem]{Example}

\newtheorem{remark}[theorem]{Remark}

\DeclareMathOperator{\Hom}{Hom}
\DeclareMathOperator{\Ann}{Ann}

\DeclareMathOperator{\Ass}{Ass}

\DeclareMathOperator{\ir}{ir}
\DeclareMathOperator{\Card}{Card}
\DeclareMathOperator{\Soc}{Soc}
\DeclareMathOperator{\Rad}{Rad}
\DeclareMathOperator{\Att}{Att}

\DeclareMathOperator{\Spec}{Spec}
\DeclareMathOperator{\p}{\frak p}
\newcommand{\fkp}{\mathfrak{p}}
\newcommand{\fkm}{\mathfrak{m}}
\newcommand{\rmr}{\mathrm{r}}

\DeclareMathOperator{\m}{\frak m}

\begin{document}
\large
\centerline{\Large {\bf REDUCIBILITY INDEX AND SUM-REDUCIBILITY INDEX}}
\vskip 1 cm
\centerline { TRAN NGUYEN AN}
\centerline{College of Education, Thai Nguyen University, Vietnam}
\centerline{E-mail: antrannguyen@gmail.com}
\vskip 0.2 cm
\centerline { TRAN DUC DUNG}
\centerline{College of Science, Thai Nguyen University, Vietnam}
\centerline{ E-mail: ducdungdhkhtn88@gmail.com}
\vskip 0.2 cm
\centerline {SHINYA KUMASHIRO}
\centerline{Department of Mathematics and Informatics, Graduate School of Science and Engineering}
\centerline{Chiba University, Yayoi-cho 1-33, Inage-ku, Chiba, 263-8522, Japan}
\centerline{E-mail: axwa4903@chiba-u.jp}
\vskip 0.2 cm
\centerline { LE THANH NHAN}
\centerline{College of Science, Thai Nguyen University, Vietnam}
\centerline{E-mail: nhanlt2014@gmail.com}
\vskip 1 cm

\noindent{\bf Abstract} {\footnote{ {\it{Key words and phrases: }}  reducibility index, flat base change,  sum-reducibility index,  Matlis duality. \hfill\break
  {\it{2000 Subject  Classification: }} 13A15, 13H10.\hfill\break {The work is supported by the Vietnam National Foundation for Science and Technology Development (Nafosted) under grant number 101.04-2017.309. The third author was supported by JSPS KAKENHI Grant Number JP19J10579.}}. 
{\normalsize
Let $R$ be a Noetherian ring. For a finitely generated $R$-module $M$,  Northcott introduced the reducibility index of $M$, which is the number of submodules appearing in an irredundant irreducible decomposition of the submodule $0$ in $M$. 
On the other hand, for an Artinian $R$-module $A$,  Macdonald proved that the number of sum-irreducible submodules appearing in an irredundant sum-irreducible representation of $A$ does not depend on the choice of the representation. This number is called the sum-reducibility index of $A$. In the former part of this paper, we compute the reducibility index of $S\otimes_R M$, where $R\to S$ is a flat homomorphism of Noetherian rings. Especially, the localization, the polynomial extension, and the completion of $R$ are studied.
For the latter part of this paper, we clarify the relation among the reducibility index of $M$,  that  of the completion of $M$, and the sum-reducibility index of the Matlis dual of $M$.
}

 \section{Introduction}


The purpose of this paper is to study about the reducibility index and sum-reducibility index. 

Let $R$ be a commutative Noetherian ring and  $M$  a finitely generated $R$-module. Let $N$ be a proper submodule of $M$.
As a fundamental result, $N$ can be expressed as an intersection of  finitely many irreducible submodules of $M$, and the number of irreducible submodules appearing in  an irredundant irreducible decomposition of $N$ is independent of the choice of decomposition (see E. Noether \cite{Noe} for the case where $M=R$). 
The number is called the  {\it  reducibility index} of $N$ in $M$ (see \cite{Nor}) and denoted by $\ir_{M}(N)$.

The study of reducibility index of finitely generated modules  has  attracted the interest of  a number of researchers.   In this topic,  they mainly pay attention to the relationship between the structure of $M$ and properties of reducibility index ${\rm ir}_M(\frak qM)$, where $R$ is local and  $\frak q$ runs over the parameter ideals of $M$. For example,  D. G.  Northcott \cite[Theorem 3]{Nor} proved that if $M$ is Cohen-Macaulay, then ${\rm ir}_M(\frak qM)$ is a constant for all  parameter ideals $\mathfrak q$ of $M$; S. Goto and   N. Suzuki  \cite{GSu} showed that  if  $M$ is generalized Cohen-Macaulay then there exists a constant $c$ such that  ${\rm ir}_M(\frak q M) \leq c$  for all parameter ideals $\frak q$ of $M$ (see also \cite{CT}). Some uniform bounds for reducibility index are given for the case where $M$ is sequentially Cohen-Macaulay and $M$ is sequentially generalized Cohen-Macaulay (see  \cite{T}, \cite{Q1}). Further extensions  are presented in \cite{Q2},  \cite{DN}.   

However, until now,  it looks that no one knows about the reducibility index under flat base changes. 
Especially, assuming that $R$ is a local ring with the maximal ideal $\fkm$, the relationship between $\ir_M(N)$ and $\ir_{\widehat M}(\widehat  N)$ has not been clarified, where $\widehat *$ denotes the $\m$-adic completion of $R$.  

The first main result of this paper reveals the reducibility index under flat base changes. Although we have formulated our result in more generality, we will restrict ourselves to studying only the case where $N=0$ since $\ir_M (N)=\ir_{M/N}(0)$ by definition. With this reason, in this paper, we  denote by $\ir_R (M)$ the reducibility index $\ir_M (0)$ of $0$ in $M$. For each $\frak p\in\Ass_R (M)$, let $ \mu_0(\p, M)$ denote the dimension of the socle of $M_{\frak p}$.
The first purpose of this paper is now stated as follows.

\begin{theorem}\label{T:1} Let $\varphi:R \to S$ be a flat homomorphism of Noetherian rings. Then
$$\ir_S (S\otimes_R M)=\sum_{\fkp\in \Ass_R (M)} \ir_S(S/\fkp S){\cdot}\mu_0(\fkp, M).$$
In addition, if $\varphi$ is faithfully flat, then 
$$ \ir_R(M)\leq \ir_S(S\otimes_R M)\leq t{\cdot}\ir_R(M),$$
where  $t:=\underset{\p \in \Ass_R(M)}{\max}  \ir (S/\p S)$. The equality  $ \ir_R(M)=\ir_S(S\otimes_R M)$ holds true if and only if $\frak p S$ is irreducible in $S$ for all $\frak p\in\Ass_R (M)$.
\end{theorem}

The second purpose of this paper is to study the sum-reducibility index  of Artinian modules. A non-zero Artinian $R$-module is said to be  {\it sum-irreducible} if it  can not be written as a sum of its two proper submodules. Let $A$ be an nonzero  Artinian $R$-module. Following I. G. Macdonald \cite{Mac}, $A$   can be  expressed as a sum of  finitely many sum-irreducible submodules of $A$, and the number of sum-irreducible submodules appearing in  an irredundant sum-irreducible representation of $A$ is independent of the choice of representation. The number is called the {\it sum-reducibility index} of $A$ and denoted by $\ir'_R(A)$.

Suppose that  $R$ is a local ring with the maximal ideal $\fkm$. Let $\widehat *$ denote  the $\m$-adic completion and $D_R(*)$ denote the Matlis dual functor. Our next problem is the relation between  $\ir_R(M)$ and $\ir'_R(D(M))$.
Let us note that, although the  reducibility index of finitely generated modules may change via the completion (Example \ref{E:2}), the sum-reducibility index of Artinian modules preserves (see Lemma \ref{L:3d}).

The following theorem gives the relation between  $\ir_R(M)$, $\ir_{\widehat R}(\widehat M)$, and $\ir'_R(D(M))$, which is the second main result of this paper.
 
  \begin {theorem} \label{T:2}  Let $(R,\m)$ be a Noetherian local ring. Then
 $$\ir_R(M)\leq \ir'_{R}(D(M))=\ir_{\widehat R}(\widehat M).$$ 
\end{theorem}

Theorem \ref{T:1}  is proved in Section \ref{section2} and Theorem \ref{T:2} is proved in Section \ref{section3}.

In what follows, unless otherwise stated, let $R$ be a commutative Noetherian ring. Let $M$ be a finitely generated $R$-module and $A$ an Artinian $R$-module. We denote by $\ell_R(M)$ the length of $M$. If $R$ is a local ring with the maximal ideal $\fkm$, denote  by $\widehat *$  the $\fkm$-adic completion and $D_R(*)$ the Matlis dual functor.

\section{Reducibility index under flat base change}\label{section2}

Throughout this section, let $\varphi: R\rightarrow S$ be a flat homomorphism of Noetherian rings. Let $M$ be a finitely generated $R$-module.  We denote by $\ir_R(M)$ the reducibility index of the zero submodule of $M$, that is the number of irreducible submodules appearing in an irredundant irreducible decomposition of the submodule $0$ of $M$.  For every prime ideal $\frak p$ of $R$, let $k(\frak p):=R_{\frak p}/\frak p R_{\frak p}$ be the residue field of $R_{\frak p}$ and set
$$ \mu_0(\p, M) = \dim_{k(\frak p)} \big(0:_{M_{\frak p}}\frak p R_{\frak p}\big),$$  the dimension of the socle of $M_{\frak p}$.  Note that $ \mu_0(\p, M)$ is  the $0$-th Bass number of $M$ with respect to $\frak p$, see \cite[page 101]{BH}.

In 1957, D. G.  Northcott  \cite{Nor} proved  that if  $(R, \frak m)$ is a Noetherian local ring and $\ell_R(M)<\infty$,  then 
$$\ir_R(M)=\dim_k\Soc (M)= \dim_k(0:_M\frak m),$$   
where $k=R/\frak m$ is the  residue field of $R$.  In general case where $M$ is not necessary of finite length and $R$ is not necessarily local,  we have the following result, see for example \cite[Lemma 2.3]{CQT}.

\begin{lemma}  \label{L:1}  
$$\ir_R(M)=\sum_{\p\in \Ass_R (M)} \mu_0(\p, M).$$
\end{lemma}

Now we prove Theorem \ref{T:1}, which is the first main result  of this paper.  

\begin{proof}[Proof of Theorem \ref{T:1}] We have by Lemma \ref{L:1} that 
$$\ir_S(S\otimes_R M)={\displaystyle \sum_{\frak Q\in \Ass_S (S\otimes_R M)} \mu_0 (\frak Q, S\otimes_R M)}.$$
Since $\varphi: R\rightarrow S$ is a flat homomorphism of Noetherian rings, we have the following relation between the set of associated primes of $M$ and that of $S\otimes_R M$, see \cite[Theorem 23.2(ii)]{Mat}
$$\Ass_S (S\otimes_R M)=\bigcup_{\frak p\in \Ass_R (M)} \Ass (S/\frak p S).$$
Suppose that $\frak p_1, \frak p_2\in\Ass_R(M)$ and $\frak Q\in \Ass (S/\frak p_1 S) \cap \Ass (S/\frak p_2S)$. By the flatness of $\varphi$, we have by  \cite[Theorem 23.2(i)]{Mat}  that  $\frak Q\cap R=\frak p_1=\frak p_2$. It follows that the union $\displaystyle \bigcup_{\frak p\in \Ass_R (M)} \Ass (S/\frak p S)$ is a disjoint union.
Hence
$$\ir_S(S\otimes_R M)={\displaystyle \sum_{\frak p\in \Ass_R (M)}}\  {\displaystyle \sum_{\frak Q\in \Ass (S/\frak p S)}} \mu_0(\frak Q, S\otimes_R M).$$

Let $\frak p\in \Ass_R (M)$ and $\frak Q\in \Ass (S/\frak p S)$, we calculate $\mu_0(\frak Q, S\otimes_R M)$. We have by \cite[Theorem 3.84]{Rot} that 
$$ \Hom_{S_{\frak Q}}(S_{\frak Q}/\frak QS_{\frak Q}, (S\otimes_R M)_{\frak Q})\cong S_{\frak Q}\otimes_S \Hom_{S}(S/\frak Q, S\otimes_R M). $$
Since $S/\frak Q =S/(\frak Q +\p S) \cong S/\frak Q \otimes S/\p S,$
$$ S_{\frak Q}\otimes_S \Hom_{S}(S/\frak Q, S\otimes_R M) \cong S_{\frak Q}\otimes_S \Hom_{S}(S/\frak Q\otimes_S S/\fkp S, S\otimes_R M).$$
By Adjoint isomorphism \cite[Theorem 2.11]{Rot},
$$ S_{\frak Q}\otimes_S \Hom_{S}(S/\frak Q\otimes_S S/\fkp S, S\otimes_R M) \cong S_{\frak Q}\otimes_S \Hom_{S}(S/\frak Q, \Hom_S (S/\fkp S, S\otimes_R M)). $$
Then
\begin{align*}
&S_{\frak Q}\otimes_S \Hom_{S}(S/\frak Q, \Hom_S (S/\fkp S, S\otimes_R M)) \\
\cong &   S_{\frak Q}\otimes_S \Hom_{S}(S/\frak Q, S\otimes_R \Hom_R (R/\fkp, M))\\
\cong & \Hom_{S_{\frak Q}}(S_{\frak Q}/\frak Q S_{\frak Q}, S_{\frak Q}\otimes_R \Hom_R (R/\fkp, M))\\
\cong & \Hom_{S_{\frak Q}}(S_{\frak Q}/\frak Q S_{\frak Q}, S_{\frak Q}\otimes_{R_\fkp} \Hom_{R_\fkp} (R_\fkp/\fkp R_\fkp, M_\fkp))
\end{align*}
by \cite[Theorem 3.84]{Rot}. On the other hand, 
$\Hom_{R_\fkp} (R_\fkp/\fkp R_\fkp, M_\fkp) $ is a finitely generated $R_\fkp/\fkp R_\fkp$-vector space of dimension $\mu_0(\fkp, M)$,
\begin{align*}
&\Hom_{S_{\frak Q}}(S_{\frak Q}/\frak Q S_{\frak Q}, S_{\frak Q}\otimes_{R_\fkp} \Hom_{R_\fkp} (R_\fkp/\fkp R_\fkp, M_\fkp)) \\
\cong & \Hom_{S_{\frak Q}}(S_{\frak Q}/\frak Q S_{\frak Q}, S_{\frak Q}\otimes_{R_\fkp} (R_\fkp/\fkp R_\fkp)^{\oplus \mu_0(\fkp, M)})\\
\cong & \Hom_{S_{\frak Q}}(S_{\frak Q}/\frak Q S_{\frak Q}, (S/\fkp S)_{\frak Q})^{\oplus \mu_0(\fkp, M)}.
\end{align*}

Hence $\mu_0(\frak Q, S\otimes_R M)=\mu_0(\frak Q, S/\frak p S){\cdot}\mu_0(\frak p, M)$. Therefore we have by Lemma \ref{L:1} that 
\begin{align} \ir_S(S\otimes_R M)&={\displaystyle \sum_{\frak p\in \Ass_R (M)}}\  {\displaystyle \sum_{\frak Q\in \Ass (S/\frak p S)}} \mu_0(\frak Q, S/\frak p S){\ }\mu_0(\frak p, M)\notag\\
&={\displaystyle \sum_{\frak p\in \Ass_R (M)}}\ \mu_0(\frak p, M){\ }\ir (S/\frak p S).\notag\end{align}
We get by the definition of $t$ and by  Lemma \ref{L:1} that
$$\ir_S(S\otimes_R M)\leq t {\displaystyle \sum_{\frak p\in \Ass_R (M)}}\ \mu_0(\frak p, M)=t \cdot\ir_R(M).$$
Now assume that $\varphi$ is faithfully flat.  Then $S/\frak p S \neq 0$ for all  $\frak p \in \Ass_R(M)$.  This implies that $\ir (S/\frak p S) \geq 1$ for all $\frak p\in\Ass_R(M).$ Therefore, we have by Lemma \ref{L:1} that
$$\ir_S(S\otimes_R M)={\displaystyle \sum_{\frak p\in \Ass_R (M)}}\ \mu_0(\frak p, M){\ }\ir (S/\frak p S)\geq {\displaystyle \sum_{\frak p\in \Ass_R (M)}}\ \mu_0(\frak p, M)=\ir_R(M).$$
In particular, the equality  $ \ir_R(M)=\ir_S(S\otimes_R M)$ holds true if and only if $\ir (S/\frak p S)=1$ for all $\frak p\in\Ass_R(M),$ if and only if $\frak p S$ is irreducible in $S$ for all $\frak p\in\Ass_R(M).$
\end{proof}

We note that the inequality $\ir_S (S\otimes_R M)\geq \ir_S(M)$ stated in Theorem \ref{T:1} does not hold without the assumption of faithfully flatness of $\varphi$.  
\begin{example} \label{E:1} {\rm Let $R$ be a Noetherian domain and $0\neq a\in R$  such that $a$ is not a unit of $R$. Let $\rm Q(R)$ denote the field of fractions of $R$. Then the natural map $R\to \rm Q(R)$ is a flat homomorphism, but not faithfully flat. Since $a$ is not unit, $R\neq aR.$ Hence $\ir (R/aR)>0$. It is clear that $\ir_{\rm Q(R)} (\rm Q(R)\otimes_R R/aR)=0$}.
\end{example}

As a consequence of Theorem \ref{T:1}, we have the following result which gives the property of the irreducibility index under localization.

\begin{corollary} \label{C:1} Let $U$ be a multiplicative closed subset of $R$. Denote by $U^{-1}R$ (resp. $U^{-1}M$) the ring of fractions of $R$ with respect to $U$ (resp. the module of fractions of $M$ with respect to $U$). Then we have
$$\ir_{U^{-1}R}(U^{-1}M)=\underset{\frak p\cap U=\emptyset}{\sum_{\frak p\in \Ass_R (M)}}\mu_0(\frak p, M).$$
In particular, $\ir_{U^{-1}R}(U^{-1}M)\leq \ir_R(M)$. The equality holds true if and only if $\frak p\cap U=\emptyset$ for all $\frak p\in\Ass_R(M).$ 
\end{corollary}
\begin{proof} 
Let $\frak p\in\Ass_R(M).$ If $\frak p\cap U\neq \emptyset$, then $U^{-1}R=\frak p U^{-1}R$. If $\frak p\cap U=\emptyset$, then $\frak p U^{-1}R$ is a prime ideal of $U^{-1}R.$ In this case, $\ir (U^{-1}R/\frak p U^{-1}R)=1.$  Now, the result follows by Theorem \ref{T:1} and Lemma \ref{L:1}.
\end{proof}

Next, we examine the irreducibility index under the polynomial extensions and formal power series extensions. 
\begin{corollary}\label{C:2} Let $R[x_1, \ldots , x_n]$ (resp. $R[[x_1, \ldots , x_n]]$) be the ring of polynomials in $n$ variables with coefficients in $R$ (resp. the ring of formal power series in $n$ variables with coefficients in $R$). Then
$$\ir (R)=\ir(R[x_1, \ldots , x_n])=\ir (R[[x_1, \ldots , x_n]]).$$
\end{corollary}
\begin{proof} Set $S:=R[x_1, \ldots , x_n].$ Let $\frak p\in\Ass (R).$ Since $S/\frak pS\cong (R/\frak p)[x_1, \ldots , x_n]$ is a domain, it follows that $\frak pS\in\Spec (S)$. Hence  $\ir (S/\frak pS)=1$. Therefore $\ir (R)=\ir (S)$ by Theorem \ref{T:1} and Lemma \ref{L:1}. The rest statement follows by the same arguments.
\end{proof}

By Theorem \ref{T:1} and Lemma \ref{L:1}, for a finitely generated $R$-module $M$, $\ir_R(M) =\ir_S(S\otimes_R~M)$ if and only if $\ir (S/\frak p S)=1$ for all $\frak p\in\Ass_R (M)$. Therefore it is natural to consider the structure of Noetherian rings with irreducibility index one. Recall that a Noetherian ring $R$ is said to be {\it generically Gorenstein if } $R_{\frak p}$ is Gorenstein for all minimal prime ideals $\frak p$ of $R$ (see \cite[page 248]{LW}). 

\begin{proposition}\label{a2.5}
Let $R$ be a Noetherian ring. Then the following conditions are equivalent.

{\rm (a)} $\ir (R)=1$.

{\rm (b)} $\Card(\Ass (R))=1$ and $R$ is generically Gorenstein.

In particular, if $R$ is a Noetherian domain, then $\ir (R)=1$. 
\end{proposition}

\begin{proof} (a) $\Rightarrow$ (b). By the assumption (a) and by Lemma \ref{L:1}, we have $$\ir (R)=\sum_{\fkp\in \Ass (R)} \mu(\fkp, R)=1.$$ Since $\mu(\fkp, R)>0$ for all $\fkp\in \Ass (R)$, it follows that $\Ass (R)$ has a single element $\frak p$. Furthermore $1=\mu(\fkp, R)=\rm r(R_{\frak p})$, thus $R_\fkp$ is Artin,  where $\rm r(*)$ denotes the Cohen-Macaulay type. Hence  $R_\fkp$ is Gorenstein, that is, $R$ is generically Gorenstein.

(b) $\Rightarrow$ (a). Set $\Ass (R)=\{ \fkp\}$.  By Lemma \ref{L:1}, we have $$\ir (R)=\mu(\fkp, R)=\ell_{R_\fkp}(\Hom_{R_\fkp}(R_\fkp/\fkp R_\fkp, R_{\fkp}))=\rmr(R_\fkp).$$  Since $R_\fkp$ is Gorenstein,  $\ir(R)=1$.
\end{proof}

\begin{corollary}
Let $S=k[x_1, x_2, \dots, x_n]$ be the polynomial ring of $n$ variables over a field $k$ and $f\in S$ a non-zero polynomial. Then $\ir (S/fS)=1$ if and only if $f=\alpha{\cdot}p^{m}$, where $\alpha\in k$ is a non-zero element, $m$ is a positive integer and $p\in S$ is an irreducible polynomial.
\end{corollary}

\begin{proof}
Since $S/fS$ is a hypersurface, $\ir (S/fS)=1$ if and only if $\Card(\Ass (S/fS))=1$. Write $f=\alpha{\cdot}p_1^{n_1}{\cdot}p_2^{n_2}\cdots {\cdot}p_{r}^{n_r}$, where $\alpha$ is a non-zero element of $k$, $p_1$, $p_2$, $\dots$, $p_r\in S$ are irreducible polynomials, and $n_1, n_2, \ldots , n_r$ are positive integers. Then $\Ass (S/fS)=\{ p_1S, p_2S, \dots, p_rS\}$, whence $r=1$ if and only if $\Card(\Ass (S/fS))=1$.
\end{proof}

From now to the end of this section, assume that $(R,\m)$ is a Noetherian local ring with the unique maximal ideal $\m$. Consider the natural faithfully flat homomorphism $R\rightarrow \widehat R,$ where $\widehat *$ denotes the $\m$-adic completion. It follows by Theorem \ref{T:1} that  $\ir_R(M)\leq \ir_{\widehat R}(\widehat M).$ The equality does not hold in general.

\begin{example}\label{E:2} 
D. Ferrand  and M. Raynaud \cite{FR}  constructed a two-dimensional Noetherian local domain $(R, \frak m)$ such that $\widehat R$ has an embedded prime $\frak Q$ of dimension $1$ (see also \cite[Section 3, Example 2]{Yo}).
Since $R$ is a domain, by Proposition \ref{a2.5}, $\ir (R)=1$. On the other hand, since $\dim R=\dim \widehat{R}=2$, there exists an associated prime $\frak Q'$ of $\widehat R$ of dimension $2$.  It follows that $\ir (\widehat{R})\geq \Card (\Ass (\widehat R))\geq 2.$
\end{example}

On the other hand,  the following holds.


\begin{proposition} \label{P:2} Suppose that $(R, \frak m)$ is an excellent local ring. For a finitely generated $R$-module $M$, the following statements are equivalent:

{\rm (a)}  $\ir_R(M)=\ir_{\widehat R}(\widehat M)$;

{\rm (b)} $\p \widehat R$ is a prime ideal of $\widehat R$ for all $\p$ in $\Ass_R(M)$.
\end{proposition}

\begin{proof} 
Let prove (a) $\Rightarrow$ (b). By Theorem \ref{T:1}, $\ir_R(M)=\ir_{\widehat R}(\widehat M)$ if and only if $\ir (\widehat{R}/\frak p \widehat{R})=1$ for all $\frak p\in\Ass_R (M)$. This is equivalent to saying that $\widehat{R}/\frak p \widehat{R}$ is generically Gorenstein and $\Card(\Ass_{\widehat{R}}(\widehat{R}/\frak p \widehat{R}))=1$ by Proposition \ref{a2.5}. On the other hand, by our assumption, $\widehat{R}/\frak p \widehat{R}$ is reduced. In fact, since $R$ is excellent, so is $R/\fkp$. Hence $\widehat{R}/\frak p \widehat{R}$ is reduced since $R/\fkp$ is domain. It follows that $\frak p\widehat{R}=\Rad (\frak p\widehat{R})$, which is a prime ideal of $\widehat R$ since $\Card(\Ass_{\widehat{R}}(\widehat{R}/\frak p \widehat{R}))=1$. The converse (b) $\Rightarrow$ (a) is now clear.
\end{proof}

\begin{example} Let $k$ be a field of characteristic zero, let $k[x]$ be the polynomial ring of one variable $x$ over $k.$ Let $R=k[x]_{(x)}$, the localization of $k[x]$ with respect to prime ideal $(x)$. Then $R$ is excellent ring,  $\widehat R =k[[x]]$, and $\Spec(R)=\{0, xR\}$, $\Spec(\widehat R)=\{0, x\widehat R\}$. It follows by Proposition \ref{P:2} that $\ir_R(M)=\ir_{\widehat R}(\widehat M)$ for all finitely generated $R$-module $M$.
\end{example}

\section{Sum-reducibility index of Artinian modules}\label{section3}

 In this section, let $R$ be a Noetherian  ring,  $M$ a nonzero finitely generated $R$-module, and $A$ a nonzero Artinian $R$-module.   

 I. G. Macdonald \cite{Mac} introduced the set of attached primes for Artinian modules, which makes an important role similarly to that of  associated primes for Noetherian modules. For given $\frak p\in\Spec (R)$,  we say that $A$ is {\it $\frak p$-secondary} if the multiplication by $a$ on $A$ is nilpotent for all $a\in \frak p$ and  surjective for all $a\in  R\setminus \frak p$. In general, $A$ admits a minimal secondary representation  $A=A_1+\ldots+A_n$, where each $A_i$ is $\p_i$-secondary. The set $\{\frak p_1,\ldots,\frak p_n\}$ is independent of the choice of the minimal secondary representation of $A$. This set is called the {\it set of  attached primes} of $A$ and denoted by $\Att_R(A)$.   

\begin{remark} \label{R:3a} {\rm  If $N$ is an irreducible submodule of $M$, then $N$ is primary. Therefore, each irredundant irreducible  decomposition of the submodule $0$ in $M$ can be reformed to a reduced primary decomposition of $0$. In particular,  $ \ir_R(M)\geq \Card (\Ass_R(M)).$  Similarly, if $B$ is a sum-irreducible submodule of $A$, then $B$ is a secondary submodule, see \cite{Mac}. Therefore,  each irredundant sum-irreducible  representation of $A$ can be reduced to a minimal secondary representation of $A$. In particular, $ \ir'_R(A)\geq \Card (\Att_R(A)).$}
\end{remark}

Next we compare the sum-reducibility index of $A$ and that of a quotient of $A$.

\begin{lemma} \label{L:3b} If  $B\subsetneq A$ is a submodule of $A$, then  $\ir'_{R}(A/B) \leq \ir'_{R}(A).$
In particular, if $A$ is a sum-irreducible, then so is $A/B$.
\end{lemma}

\begin{proof}  Firstly, assume that $A$ is sum-irreducible. If $A/B=C_1/B+C_2/B,$ where $C_1, C_2$ are submodules of $A$ containing $B$, then $A=C_1+C_2$, and hence $A=C_i$ for some $i\in\{1, 2\}.$ Hence $A/B$ is sum-irreducible.

Now assume that $\ir'_{R}(A)=n$. Let $A=A_1+\ldots +A_n$ be an irredundant sum-irreducible representation of $A$.  Then we have
$$A/B = (A_1+B)/B+\ldots+(A_n+B)/B.$$
 For each $i$, if $A_i\not\subseteq B$, then  $(A_i+B)/B\cong A_i/(A_i\cap B)$ is  sum-irreducible.  Thus, from the above representation, we can reduce to an irredundant sum-irreducible representation of $A/B$ with at most $n$ components. Hence  $\ir'_{R}(A/B) \leq n$. 	
\end{proof}

For $\p \in \Att_{R}(A)$, let $\Lambda'_{\mathfrak{p}}(A)$ denote the set of all $\p$-secondary submodules of $A$ which appear in a minimal secondary representation of $A$.  It follows by \cite{Mat} that if $\frak p\in\min\Att_R(A)$, then $\Lambda'_{\mathfrak{p}}(A)$ has a unique element. Suppose $\p\in\Att_R(A)$ is an embedded attached prime.  A submodule $B\in \Lambda'_{\mathfrak{p}}(A)$ is said to be  a $\mathfrak{p}$-minimal embedded component of $A$  if  $B$ is a minimal element in the set $\Lambda'_{\mathfrak{p}}(A)$ (under the inclusion).  If  ${\rm{Att}}_R(A) = \{{\mathfrak{p}}_1,\ldots,{\mathfrak{p}}_n\}$ and $B_i \in \Lambda'_{\mathfrak{p}_i}(A)$ for $i=1,\ldots,n$, then  $A=B_1+ \ldots +B_n$ is a minimal  secondary representation of $A$, see  \cite[Theorem 4.1.2]{Y}. 
The following result  is an analogue  of  \cite [Theorem 3.2]{CQT} on the reducibility index of finitely generated modules. 

 \begin{proposition} \label{P:3c}  Let $A=B_1+\ldots+ B_n$ be a minimal secondary representation of $A$, where $B_i$ is $\mathfrak{p}_i$-secondary for all $i=1,\ldots, n.$ If $B_i$ is a $\mathfrak{p}_i$-minimal embedded component of $A$ for all embedded attached prime ideals $\frak p_i$ of $A$, then 
$$\ir'_R(A)=\ir'_R(B_1)+\ldots+\ir'_R(B_n).$$
 \end{proposition} 

  \begin{proof} For  $i\in\{1, \ldots , n\},$ set  $t_{i}=\ir'_{R}(B_{i}) $.  Let $B_{i} = B_{i1}+ \ldots + B_{it_{i}}$ be an irredundant sum-irreducible representation of $B_{i}$. Then 
$$A=\sum_{i=1}^n  (B_{i1}+ \ldots + B_{it_{i}})$$
is a sum-irreducible representation of  $A$. Suppose in contrary that this  representation is redundant. Without loss of any generality, we can assume that $B_{11}$  is redundant.  So, we have 
$$A= (B_{12}+ \ldots +B_{1t_1})+\sum_{i=2}^n  (B_{i1}+ \ldots + B_{it_{i}}).$$
Set $B'_1=B_{12}+ \ldots +B_{1t_1}.$ We claim that $B'_1$ is $\frak p_1$-secondary.  In fact, we note that $B'_1\neq 0$, since the secondary representation $A=B_1+\ldots+ B_n$ is minimal.  For each $j=2, \ldots , t_1,$ since $B_{1j}$ is  sum-irreducible, it is secondary. Set $\frak q_j=\Rad(\Ann_RB_{1j}).$ Then $\frak q_j^tB_{1j}=0$ for some positive integer $t.$  Moreover, $\frak q_j\supseteq \Rad (\Ann_RB_1)=\frak p_1.$ If $\frak q_j\neq \frak p_1$, then $$B_1=\frak q_j^tB_1=\frak q_j^t(B_{11}+\ldots +B_{1t})\subseteq B_{11}+\ldots + B_{1(j-1)}+B_{1(j+1)}+\ldots +B_{1t_1}.$$  This is impossible.   Therefore $\frak q_j=\frak p_1$ for all $j.$ Hence $B'_1$ is $\frak p_1$-secondary, and the claim is proved. By the claim,  $A=B'_1+B_2+\ldots +B_n$ is a minimal secondary representation of $A$. Hence $B_1$ is not a $\mathfrak{p}_1$-minimal embedded component of $A$. This gives a contradiction. Hence, the representation $A=\sum_{i=1}^n  (B_{i1}+ \ldots + B_{it_{i}})$ is irredundant. It means that
 	$$\ir'_{R}(A) = \ir'_{R}(B_{1}) + \ldots + \ir'_{R}(B_{n}).$$	
 \end{proof}

From now on, assume that $(R,\frak m)$ is a Noetherian local ring, $k$ denotes the residue field $R/\frak m$. Note that $A$  has a natural structure as an Artinian $\widehat R$-module. With the $\widehat R$-module structure,   a subset  of $A$ is an $R$-submodule if and only if it is an $\widehat R$-submodule of $A$, see \cite[8.2.4, 10.2.18]{BS}.  Therefore, each irredundant sum-irreducible representation of $R$-submodule $A$ is  an irredundant  sum-irreducible representation  of $\widehat R$-submodule $A$.  It follows that the sum-reducibility  index of Artinian modules  is preserved under $\frak m$-adic completion. Note that the reducibility  index of finitely generated modules is not necessarily preserved under $\frak m$-adic completion, see Example \ref{E:2}. 

  \begin{lemma} \label{L:3d} $\ir'_R(A)=\ir'_{\widehat R}(A).$
 \end{lemma}

D. G.  Northcott  \cite{Nor} proved that   if  $\ell_R(M)<\infty$, then 
$\ir_R(M)=\dim_k\Soc (M).$ The following lemma gives an analogue to this result. 

\begin{lemma}  \label{L:3e} If $\ell_R(A)<\infty$, then $\ir'_R(A)=\dim_k(A/\frak mA).$
\end{lemma}
\begin{proof}  Set $\dim_k (A/\frak mA)=n.$ Let $\{\overline e_1, \ldots , \overline e_n\}$ be a basis of the $k$-vector space $A/\frak mA$, where $\overline{e_i}$ denotes the image of $e_i \in A$ into $A/\frak m A$ for $i=1, \ldots , n.$  Since $\ell_R(A)<\infty$,  we get that $\{e_1, \ldots , e_n\}$ is a minimal system of generators of $A$.  For  each $i\in\{1, \ldots ,n\},$ set $B_i=Re_i$. Then $B_i$ is not redundant in the representation $A=B_1+\ldots +B_n$.  For each $i$, suppose that  $B_i=C+D,$ where $C, D$ are submodules of $B_i.$ Then $e_i=x+y$ for some $x\in C$ and $y\in D.$ Hence  $\{e_i\}$ and $\{x, y\}$ are systems of generators of $B_i$.  Therefore $\{x, y\}$ is not a minimal system of generators of $B_i$, see  \cite[Theorem 2.3(i)]{Mat}. Hence $B_i=Rx$ or $B_i=Ry,$ i.e. $B_i=C$ or $B_i=D$.  Therefore $B_i$ is sum-irreducible.
\end{proof} 

 Let $E=E(R/\frak m)$ be the injective hull of $R/\frak m$. Denote by $D_R(-)=\Hom_R(-, E)$ the Matlis duality functor. Since $M$ is a finitely generated, $D_R(M)$ is an Artinian $R$-module. Therefore, it is natural to ask about a relationship between $\ir_R(M)$ and $\ir'_R(D_R(M)).$ Before giving the answer for the case where $\ell_R(M)<\infty$, let us note the following.

\begin{lemma}  \label{L:3g} 
$\dim_k\Soc (M)=\dim_k(D_R(M)/\frak m D_R(M)).$
\end{lemma}
The following corollary follows by Lemma \ref{L:3e} and Lemma \ref{L:3g}.
\begin{corollary}  \label{C:3g} If $\ell_R(M)<\infty$, then $\ir_R(M)=\ir'_R(D_R(M)).$
\end{corollary}

Our goal in this section is to compare $\ir_R(M)$ and $\ir'_R(D_R(M))$  in general case where $M$ is not necessary of finite length. We need the following lemmas.

Let $B$  be a submodule of $A$, then the canonical projection $p: A\rightarrow A/B$ induces the injection $D_R(A/B)\rightarrow D_R(A)$, where an element $f\in D_R(A/B)$ can be identified with the element $fp\in D_R(A).$  Thus, we consider $D_R(A/B)$ to be  a submodule of $D_R(A).$

\begin{lemma} \label{L:3h} Suppose that $B, C$  are submodules of $A$.  Then $D_R(A/B) \cap D_R(A/C) = 0$ if and only if $A=B+C$.
\end{lemma} 
	
\begin{proof}  Let $D_R(A/B)\cap D_R(A/C)=0.$ Since $D_R(A/(B+C))\subseteq D_R(A/B)\cap D_R(A/C)$, we have $D_R(A/(B+C))=0.$  Hence $A=B+C.$

Let $A=B+C$ and $f \in D_R(A/B) \cap D_R(A/C)$. Then we can write $f=f_1p_1=f_2p_2$ for some $f_1: A/B \rightarrow E(R/\frak m)$, $f_2: A/C  \rightarrow E(R/\frak m)$, where  the maps $p_1:A/(B\cap C) \rightarrow A/B$ and $p_2: A/(B\cap C) \rightarrow A/C$ are natural projections. Let $a \in A$. Since $A=B+C$, we can write  $a=b+c$ with $b \in B$ and $c \in C$. Then
	$$f(a+B \cap C)=f(b+c+ B \cap C) = f_1(c+B)=f_2(b+C).$$
	Since $f_1(c+B)=f(c+B \cap C)$ and $f_2(b+C)=f(b+B \cap C)$, we have 	$f(c+B \cap C) = f(b+B \cap C)$. Hence $f(b-c+B \cap C)=0$. So, $f_2(b+C)=0$.  Hence $f(a+B \cap C)=0$. Thus $f=0$.
	\end{proof}

For an Artinian $R$-module $A$, since $A$ has a natural structure as an Artinian $\widehat R$-module, $D_R(A)\cong D_{\widehat R}(A)$ which is a finitely generated $\widehat R$-module.

\begin{lemma} \label{L:3i}  The following statements are true.
\begin{itemize}
\item[\rm{(a)}]   The submodule $0$ is  irreducible in $\widehat M$ if and only if $D_R(M)$ is sum-irreducible.
\item[\rm{(b)}]   $A$ is  sum-irreducible  if and only if the submodule $0$ is  irreducible in $\widehat R$-module $D_R(A).$
\end{itemize}
\end{lemma}

\begin{proof}  (a) Suppose that  $0$ is  irreducible in $\widehat M$. Note that $D_R(D_R(M))\cong \widehat M$. Therefore, for any  proper submodules $B, C$ of $D_R(M)$, we can identify  $D_R(D_R(M)/B)$ and $D_R(D_R(M)/C)$ to be non-zero submodules of $\widehat M$. Since $0$ is  irreducible in $\widehat M$, we have $$D_R(D_R(M)/B) \cap D_R(D_R(M)/C)\neq 0.$$ Hence $D_R(M)\neq B+C$ by Lemma \ref{L:3h}.  Thus, $D_R(M)$ is sum-irreducible.

Conversely, assume that $D_R(M)$ is sum-irreducible. Let $L, P$ be non-zero submodules of $\widehat M$. Since $\widehat M\cong D_R(D_R(M))$ and $D_R(M)\cong D_R(D_R(D_R(M)))$,  it follows that $D_R(L)$ and $D_R(P))$ are quotients of $D_R(M)$.  Set $D_R(L)=D_R(M)/B$ and $D_R(P)=D_R(M)/C$, where $B$ and $C$ are submodules of $D_R(M)$. Because $L, P\neq 0$, we have $B, C\neq D_R(M).$ Since $D_R(M)$ is sum-irreducible, $D_R(M)\neq B+C.$ Hence $D_R(D_R(M)/B)\cap D_R(D_R(M)/C)\neq 0$ by Lemma \ref{L:3h}. It follows that $L\cap P\neq 0.$ Thus, the submodule $0$ is  irreducible in $\widehat M$.
 
(b) Suppose that $A$ is  sum-irreducible. Then $A$ is also  sum-irreducible as an $\widehat R$-module.  Since $A\cong D_R(D_R(A))$, it follows by (a) that $0$ is  irreducible in $\widehat R$-module $D_R(A).$ 

Conversely, suppose that $0$ is  irreducible in $\widehat R$-module $D_R(A).$ It follows by (a) that $D_R(D_R(A))$ is sum-irreducible. It means that $A$ is sum-irreducible.
\end{proof}

Now, we are ready to prove Theorem \ref{T:2}, which is the second main result of this paper.

\begin{proof}[Proof of Theorem \ref{T:2}]  Note that $D_R(M)\cong D_{\widehat R}(\widehat M)$ as Artinian $\widehat R$-modules. Therefore, by Theorem \ref{T:1} and Lemma \ref{L:3d}, it is enough to prove $\ir_R(M)=\ir'_{R}(D_R(M))$ under the assumption that $R=\widehat R.$  

Set $\ir_{R}(M)=t$. Let $0={\bigcap_{i=1}^{t}}N_i$ be an  irredundant irreducible decomposition of submodule $0$ of $M$. We prove $\ir_R(M)=\ir'_{R}(D_R(M))$ by induction on $t$. The case where $t=1$ follows by Lemma \ref{L:3i}(a). Let $t>1$ and assume that the result is true for $t-1.$
Set $N=\bigcap_{i=2}^tN_i.$ Consider the exact sequence
$$0\rightarrow M\overset{f}{\rightarrow} M/N_1\oplus M/N\overset{g}{\rightarrow} M/(N_1+ N)\rightarrow 0,$$
where $f(x)=(x+N_1, x+N)$ for all $x\in M$, and $g(x+N_1, y+N)=x-y+(N_1+N)$ for all $x+N_1\in M/N_1$, $y+N\in M/N.$ Therefore, we get the induced exact sequence
$$0\rightarrow D_R(M/(N_1+ N))\overset{g^*}{\rightarrow} D_R(M/N_1)\oplus D_R(M/N)\overset{f^*}{\rightarrow} D_R(M)\rightarrow 0.$$
Note that $g=q_1p_1+q_2p_2$, where
$p_1: M/N_1\oplus M/N\rightarrow M/N_1$ and $p_2: M/N_1\oplus M/N\rightarrow M/N$ are natural projections, $q_1: M/N_1\rightarrow M/(N_1+N)$ is defined by $q_1(x+N_1)=x+(N_1+N)$, and $q_2: M/N\rightarrow M/(N_1+N)$ is defined by $q_2(x+N)=-x+(N_1+N)$. Therefore, 
\begin{align}D_R(M)&\cong \big(D_R(M/N_1)\oplus D_R(M/N)\big)/{\rm Im} (g^*)\notag\\
&\cong \Big(D_R(M/N_1)/{\rm Im} (q_1^*)\Big)\oplus \Big(D_R(M/N)/{\rm Im} (q_2^*)\Big),\notag\end{align} where $q_1^*, q_2^*$ are respectively the maps induced by $q_1, q_2$ under Matlis dual functor. From the exact sequence 
$$0\rightarrow (N_1+N)/N_1\rightarrow M/N_1\overset{q_1}{\rightarrow}M/(N_1+N)\rightarrow 0,$$
we get the exact sequence
$$0\rightarrow D_R(M/(N_1+N))\overset{q_1^*}{\rightarrow} D_R(M/N_1)\rightarrow D_R((N_1+N)/N_1)\rightarrow0.$$
Note that $(N_1+N)/N_1\cong N$ and $N\neq 0$. Therefore,  $N_1$ is irreducible in $N_1+N$. So, from the above exact sequence we have by Lemma \ref{L:3i}(a) that
$$\ir'_R\big(D_R(M/N_1)/{\rm Im} (q_1^*)\big)=\ir'_R\big(D_R((N_1+N)/N_1)\big)=1.$$
From the exact sequence 
$$0\rightarrow (N_1+N)/N\rightarrow M/N\overset{q_2}{\rightarrow}M/(N_1+N)\rightarrow 0,$$
we get the exact sequence
$$0\rightarrow D_R(M/(N_1+N))\overset{q_2^*}{\rightarrow} D_R(M/N)\rightarrow D_R((N_1+N)/N)\rightarrow0.$$
 Set $N_i'=N_1\cap N_i$ for all $i=2, \ldots , t.$ Then we have
$$(N_1+N)/N\cong N_1\cong N_1/\overset{t}{\underset{i=2}{\cap}}N_i'.$$
Let $i\in\{2, \ldots , t\}.$ Note that $N_1/N_i'\cong (N_1+N_i)/N_i.$ Moreover,  $N_i$ is irreducible in $N_1+N_i.$ Therefore, we have
$\ir_R(N_1/N_i')=\ir_R((N_1+N_i)/N_i)=1.$ Hence, $N_i'$ is irreducible in $N_1$ for all $i=2, \ldots , t.$ It follows that $0=\overset{t}{\underset{i=2}{\cap}}N_i'$ is an irredundant irreducible decomposition of the submodule $0$ in $N_1.$ Hence $$\ir_R(N_1+N)/N=\ir_R(N_1)=t-1.$$ So, we get by induction that 
$$\ir'_R\big(D_R(M/N)/{\rm Im} (q_2^*)\big)=\ir'_R\big(D_R((N_1+N)/N)\big)=t-1.$$
Thus, we have $$\ir'_R\big(D_R(M))=\ir'_R\big(D_R(M/N_1)/{\rm Im} (q_1^*)\big)+\ir'_R\big(D_R(M/N)/{\rm Im} (q_2^*)\big)=t.$$
\end{proof}

Finally, the following example  clarifies the result in Theorem \ref{T:2} in case where $R$ is not complete under the $\frak m$-adic topology.

\begin{example} Consider the  Noetherian local domain $(R, \frak m)$ constructed by D. Ferrand  and M. Raynaud \cite{FR} such that $\widehat R$ has an embedded prime of dimension $1$. Then $\ir(R)=1$. We have $D_R(R)\cong E(R/\frak m),$ the injective hull of $R/\frak m$.  We get by Theorem \ref{T:2} and by Example \ref{E:2} that $$\ir'_R(D_R(R))=\ir'_R (E(R/\frak m))=\ir (\widehat R)>1.$$
\end{example}

\end{document}